\documentclass[reqno]{amsart}
\usepackage{amsmath, amssymb, bm, url, enumitem}
\usepackage[a4paper, margin=2cm]{geometry}
\usepackage{tikz}
\usetikzlibrary{arrows,arrows.meta,decorations.pathreplacing,decorations.markings,shapes,calc,cd}
\tikzset{labelsize/.style={font=\scriptsize}}
\usepackage[colorlinks=true,pagebackref]{hyperref}
\usepackage{cleveref}
\hypersetup{%
    colorlinks=true,
    linkcolor=blue,
    citecolor=blue
}

\newtheorem{theorem}{Theorem}[section]
\newtheorem{lemma}[theorem]{Lemma}
\newtheorem{proposition}[theorem]{Proposition}
\newtheorem{corollary}[theorem]{Corollary}

\theoremstyle{definition}
\newtheorem{definition}[theorem]{Definition}
\newtheorem{example}[theorem]{Example}

\theoremstyle{remark}
\newtheorem{remark}[theorem]{Remark}
\numberwithin{equation}{section}

\crefname{theorem}{Theorem}{Theorems}
\crefname{section}{Section}{Sections}
\crefname{subsection}{Subsection}{Subsections}
\crefname{definition}{Definition}{Definitions}
\crefname{notation}{Notation}{Notations}
\crefname{example}{Example}{Examples}
\crefname{remark}{Remark}{Remarks}
\crefname{equation}{}{}
\crefname{corollary}{Corollary}{Corollaries}
\crefname{proposition}{Proposition}{Propositions}
\crefname{lemma}{Lemma}{Lemmas}
\crefname{figure}{Figure}{Figures}

\newcommand{\Z}{\mathbb{Z}}
\newcommand{\R}{\mathbb{R}}
\newcommand{\Hom}{\mathrm{Hom}}
\newcommand{\Conn}{\mathrm{Conn}}
\newcommand{\Aut}{\mathrm{Aut}}
\newcommand{\norm}[1]{\left\lVert#1\right\rVert_1}

\title{Homotopy types of Hom complexes of graph homomorphisms whose codomains are cycles}

\author[S. Fujii]{Soichiro Fujii}
\address{School of Mathematical and Physical Sciences, Macquarie University, NSW 2109, Australia\vspace{-0.6em}}
\address{
Department of Mathematics and Statistics, Masaryk University, Kotl\'a\v{r}sk\'a 2, 611 37 Brno, Czech Republic}
\email{s.fujii.math@gmail.com}

\author[Y. Iwamasa]{Yuni Iwamasa}
\address{
Graduate School of Informatics,
Kyoto University \\
Yoshida Honmachi, Sakyo-ku, Kyoto, 606-8501 \\
Japan
}
\email{iwamasa@i.kyoto-u.ac.jp}

\author[K. Kimura]{Kei Kimura}
\address{
Faculty of Information Science and Electrical Engineering, Kyushu University \\ 744, Motooka, Nishi-ku, Fukuoka, 819-0395, Japan}
\email{kkimura@inf.kyushu-u.ac.jp}

\author[Y. Nozaki]{Yuta Nozaki}
\address{
Faculty of Environment and Information Sciences, Yokohama National University \\
79-7 Tokiwadai, Hodogaya-ku, Yokohama, 240-8501 \\
Japan\vspace{-0.6em}}
\address{
WPI-SKCM$^2$, Hiroshima University \\
1-3-1 Kagamiyama, Higashi-Hiroshima, Hiroshima, 739-8526 \\
Japan}
\email{nozaki-yuta-vn@ynu.ac.jp}

\author[A. Suzuki]{Akira Suzuki}
\address{
Center for Data-driven Science and Artificial Intelligence, Tohoku University \\
41 Kawauchi, Aoba-ku, Sendai, 980-8579 \\
Japan
}
\email{akira@tohoku.ac.jp}


\begin{document}
\begin{abstract}
For simple graphs $G$ and $H$, the Hom complex $\mathrm{Hom}(G,H)$ is a polyhedral complex whose vertices are the graph homomorphisms $G\to H$ and whose edges connect the pairs of homomorphisms which differ in a single vertex of $G$. 
Hom complexes play an important role in an algebro-topological approach to the graph coloring problem.
It is known that $\mathrm{Hom}(G,H)$ is homotopy equivalent to a disjoint union of points and circles when both $G$ and $H$ are cycles.
We generalize this known result by showing that
the same holds whenever $G$ is connected and $H$ is a cycle.
To this end, we explicitly construct the universal cover of each connected component of $\mathrm{Hom}(G,H)$ and prove that it is contractible.
Additionally, we provide a simple criterion to determine whether the connected component containing a given homomorphism is homotopy equivalent to a point or circle.
\end{abstract}
\maketitle

\setcounter{tocdepth}{1}
\tableofcontents

\section{Introduction}
\label{sec:Introduction}
In this paper, by a \emph{graph} we mean an undirected simple graph, assumed to be finite unless otherwise stated.
For graphs $G$ and $H$, a \emph{graph homomorphism} $G\to H$ is a function $V(G)\to V(H)$ between the sets of vertices preserving the adjacency relation.
The \emph{Hom complex} $\Hom(G,H)$ (see, e.g.,~\cite[Section~9.2.3 and Chapter~18]{Koz08}) is a certain polyhedral complex having the homomorphisms $G\to H$ as the vertices.
Its $1$-skeleton $\Hom(G,H)^{(1)}$ is the graph with the homomorphisms $G\to H$ as the vertices, and in which two homomorphisms $f,g\colon G\to H$ are adjacent if and only if there exists $u\in V(G)$ such that $f(u)\neq g(u)$ and $f(v)=g(v)$ for all $v\in V(G)\setminus\{u\}$. 
The reachability problem in graphs of the form $\Hom(G,H)^{(1)}$ has been studied (in the context of combinatorial reconfiguration~\cite{ito2011}) from an algorithmic point of view by Wrochna~\cite{Wro20}, among others. The notion of a Hom complex has been generalized to directed graphs by Dochtermann and Singh~\cite{DoSi23}, who propose it as a natural framework for reconfiguration of homomorphisms of directed graphs.

The graph homomorphisms from $G$ to the complete graph $K_k$ with $k$ vertices correspond to the vertex colorings of $G$ with $k$ colors, or the \emph{$k$-colorings} of $G$ for short. 
Since a homomorphism $G\to K_k$ induces a continuous map $\Hom(T,G)\to \Hom(T,K_k)$ for any graph $T$, topological invariants of $\Hom(T,G)$ and $\Hom(T,K_k)$ can be used to give an obstruction for the existence of $k$-colorings of $G$.
Lov\'{a}sz~\cite{Lov78} used the neighborhood complex of $G$, which is homotopy equivalent to $\Hom(K_2,G)$, to solve Kneser's conjecture about
the (non-)existence of $k$-colorings for a certain class of graphs called Kneser graphs.
Related results in this direction can be found in \cite{BaKo07} and \cite{Mat17JMSUT}, for example.

The homotopy type of $\Hom(G,H)$ has been determined in some cases.
For instance, Babson and Kozlov~\cite{BaKo06} show that $\Hom(K_n,K_k)$ is homotopy equivalent to a wedge of $(k-n)$-spheres whenever $n \leq k$ (otherwise, $\Hom(K_n,K_k)$ is empty).
\v{C}uki\'{c} and Kozlov~\cite{CuKo06} prove that $\Hom(C_n,C_k)$ is homotopy equivalent to a disjoint union of points and circles, where $C_k$ denotes the cycle with $k$ vertices for $k\geq 3$.
The main result of this paper is the following generalization of the above-mentioned result of \cite{CuKo06}.

\begin{theorem}
\label{thm:HomGCk}
Let $G$ be a connected graph and $k\geq 3$.
Then the Hom complex $\Hom(G,C_k)$ is homotopy equivalent to a disjoint union of points and circles.
\end{theorem}

Combining \cref{thm:HomGCk} with the property $\Hom(G\sqcup G',H)\cong \Hom(G,H)\times \Hom(G',H)$ \cite[Section~2.4]{BaKo06}, we obtain the following consequence.

\begin{corollary}
Let $G$ be a \textup{(}not necessarily connected\textup{)} graph and $k\geq 3$.
Then the Hom complex $\Hom(G,C_k)$ is homotopy equivalent to a disjoint union of products of circles.\footnote{Note that the product $(S^1)^d$ of $d$ circles is a $d$-dimensional torus. When $d=0$, we define $(S^1)^0$ as a point.}
\end{corollary}

As in \cite{CuKo06}, we take different approaches to proving \cref{thm:HomGCk}, depending on whether $k=4$ or $k\neq 4$. When $k=4$, we can use the folding theorem \cite[Theorem~18.22]{Koz08}; in this case, $\Hom(G,C_4)$ is homotopy equivalent to $\Hom(G,K_2)$, and thus each connected component turns out to be contractible.

When $k\neq 4$, the homomorphisms $G\to C_k$ have the ``monochromatic neighborhood property'' \cite{Wro20,LMS25}: for any pair of homomorphisms $f,g\colon G\to C_k$ which are adjacent in $\Hom(G,C_k)^{(1)}$, if $u\in V(G)$ is the vertex with $f(u)\neq g(u)$, then all vertices of $G$ adjacent to $u$ take the same value under $f$ (or equivalently, under $g$); see \cref{prop:MNP}. Using this, we can describe the universal cover $p\colon \widetilde X\to X$ of each connected component $X$ of $\Hom(G,C_k)$ fairly explicitly, and show that $\widetilde X$ is contractible. 
This implies that $X$ is an Eilenberg--MacLane space $K\bigl(\pi_1(X),1\bigr)$ (see, e.g.,~\cite[Section~1.B]{Hat02}), where $\pi_1(X)$ is the fundamental group of $X$. 
We show that $\pi_1(X)$ is either trivial or is isomorphic to the additive group $\Z$ of integers, therefore showing (by the uniqueness of the homotopy type of an Eilenberg--MacLane space) that $X$ is either contractible or is homotopy equivalent to the circle.
Our calculation of $\pi_1(X)$ also takes advantage of the explicit description of the universal cover $p\colon \widetilde X\to X$: since $\pi_1(X)$ is isomorphic to the automorphism group $\Aut(p)$ of the universal cover $p$ of $X$, we calculate the latter.

\subsection*{Related work}
A few months after the first arXiv version of this paper had been made public, Fujii--Kimura--Nozaki \cite{FKN25} and Matsushita \cite{Mat25} have independently proved the following result.\footnote{The fact that \cite{FKN25} was being developed was mentioned in the first arXiv version of this paper.} 
Let $G$ be a connected graph and $H$ a square-free graph, i.e., a graph not containing $C_4$ as a subgraph. 
Then, each connected component of $\Hom(G,H)$ is homotopy equivalent to a point, a circle, or (the geometric realization of) a connected component of $H \times K_2$ (\cite[Theorem~1.1]{FKN25} and \cite[Theorem~1.2]{Mat25}). 
In particular, since a connected component of $C_k \times K_2$ ($k \geq 3$) is also a cycle (more precisely, it is $C_k$ if $k$ is even and $C_{2k}$ if $k$ is odd), their results imply \cref{thm:HomGCk} on $\Hom(G,C_k)$ with $k \neq 4$. 
Moreover, the high-level strategies of this paper and \cite{FKN25,Mat25} are similar, in the sense that 
they all construct the universal cover $p\colon \widetilde X\to X$ of each connected component $X$ of $\Hom(G,H)$ and show that $\widetilde X$ is contractible.

However, there are a few notable differences between this paper and \cite{FKN25,Mat25}. 
First, in this paper, we construct $\widetilde X$ rather concretely as a suitable induced cube subcomplex of $\R^n$ (see \cref{def:G-L} and \cref{prop:E-alpha-as-Conn}). 
Such a concrete construction does not seem possible in the general setting of \cite{FKN25,Mat25}, and the authors there construct it more abstractly (using poset topology).
As a result, the details of the proofs in this paper differ significantly from those in \cite{FKN25,Mat25}.
In particular, unlike \cite{FKN25,Mat25}, this paper does not rely on the theory of poset topology.
Second, as mentioned in \cite[Related work]{FKN25}, the algorithm to determine the homotopy type of a given connected component of $\Hom(G,H)$ given in this paper when $H=C_k$ (\cref{thm:homotopy-type-of-Hom}) is far simpler than its counterpart for a general square-free $H$, which is essentially the algorithm of Wrochna \cite{Wro20}; see \cite[Remark~6.9]{FKN25} or \cite[Corollary~5.6]{Mat25}.
Accordingly, \cref{thm:homotopy-type-of-Hom}, which sharpens \cref{thm:HomGCk}, is \emph{not} encompassed by the generalization in \cite{FKN25,Mat25}.

Therefore, this paper provides a concrete and more accessible analysis of $\Hom(G,H)$ when $H$ is a cycle, 
and also demonstrates that an algorithm for determining the homotopy type admits a substantial simplification in this case.

\subsection*{Outline of the paper}
In \cref{sec:Preliminaries}, after recalling relevant background information and introducing key definitions in this paper, we treat the special cases of \cref{thm:HomGCk} where $k=4$ (\cref{prop:square}) or where $G$ consists of a single vertex (\cref{prop:isolated}). The rest of this paper deals with the case where $G$ is a connected graph with at least two vertices and where the integer $k\geq 3$ satisfies $k\neq 4$. We fix an arbitrary homomorphism $f\colon G\to C_k$; our aim is to determine the homotopy type of the connected component $\Conn\bigl(\Hom(G,H),f\bigr)$ of $\Hom(G,H)$ containing $f$.
In \cref{sec:Construction}, we construct a covering map $p_f\colon E_f\to \Conn\bigl(\Hom(G,H),f\bigr)$. In \cref{sec:Contractibility}, we show that $E_f$ is contractible. This in particular implies that $E_f$ is simply-connected, i.e., that $p_f\colon E_f\to \Conn\bigl(\Hom(G,H),f\bigr)$ is the universal cover of $\Conn\bigl(\Hom(G,H),f\bigr)$. 
Finally, in \cref{sec:Determining}, we determine the homotopy type of $\Conn\bigl(\Hom(G,H),f\bigr)$ by calculating the fundamental group of $\Conn\bigl(\Hom(G,H),f\bigr)$, or equivalently, the automorphism group $\Aut(p_f)$ of the universal cover $p_f$ (\cref{thm:homotopy-type-of-Hom}).

\subsection*{Acknowledgments}
The authors thank Takahiro Matsushita for the helpful discussion.
The authors are grateful to the anonymous referees for a careful reading of the manuscript and for helpful comments.
This study was supported in part by JSPS KAKENHI Grant Numbers JP20K14317, JP21K17700, JP22K17854, JP23K12974, JP24H00686, JP24K02901, JP24K21315, JP25K14980, JSPS Overseas Research Fellowships, and JST ERATO Grant Number JPMJER2301, Japan.

\section{Preliminaries}
\label{sec:Preliminaries}

For a graph $G$, its set of vertices is denoted by $V(G)$ and its set of edges by $E(G)$; the latter is a subset of the set of all two-element subsets of $V(G)$, whose element $\{u,v\}$ we write as $uv$. 
(Recall that by a graph we mean an undirected simple graph.)
Two vertices $u,v\in V(G)$ of $G$ are \emph{adjacent} if $uv\in E(G)$. A vertex of $G$ is \emph{isolated} if it is not adjacent to any vertex of $G$. A graph $G$ is \emph{finite} if $V(G)$ is a finite set. We assume that all graphs are finite unless otherwise stated. 

For a digraph $\overrightarrow{G}$, its set of vertices is denoted by $V(\overrightarrow{G})$ and its set of arcs by $A(\overrightarrow{G})$, the latter being a subset of $V(\overrightarrow{G})^2$.
A vertex $u\in V(\overrightarrow{G})$ is called a \emph{source} in $\overrightarrow{G}$ if there is no vertex $v\in V(\overrightarrow{G})$ such that $(v,u) \in A(\overrightarrow{G})$, and a \emph{sink} in $\overrightarrow{G}$ if there is no vertex $v\in V(\overrightarrow{G})$ such that $(u,v) \in A(\overrightarrow{G})$.

Let $G$ and $H$ be possibly infinite graphs. A \emph{graph homomorphism} $f\colon G\to H$ is a function $f\colon V(G)\to V(H)$ such that, for any pair of vertices $u,v\in V(G)$, if $u$ and $v$ are adjacent in $G$, then so are $f(u)$ and $f(v)$ in $H$.  
A graph homomorphism $f\colon G\to H$ (between possibly infinite graphs) is said to have
\begin{itemize}
    \item the \emph{edge-lifting property} if for each vertex $u\in V(G)$ and each vertex $y\in V(H)$ such that $f(u)$ and $y$ are adjacent in $H$, there exists a vertex $v\in V(G)$ such that $u$ and $v$ are adjacent in $G$ and $f(v)=y$, and
    \item the \emph{unique edge-lifting property} if for each vertex $u\in V(G)$ and each vertex $y\in V(H)$ such that $f(u)$ and $y$ are adjacent in $H$, there exists a \emph{unique} vertex $v\in V(G)$ such that $u$ and $v$ are adjacent in $G$ and $f(v)=y$.
\end{itemize} 

The Hom complex $\Hom(G,H)$ for graphs $G$ and $H$ is a certain topological space defined in, e.g., \cite[Section~9.2.3]{Koz08}.
For our purposes, it is convenient to describe $\Hom(G,H)$ using the fact that it is a \emph{flag prodsimplicial complex}. 
A \emph{prodsimplicial complex} (cf.~\cite[Definition~2.43]{Koz08}) is a regular CW complex $K$ whose face poset $\mathcal{F}(K)$ has the following property: for any cell $C$ in the complex $K$, the subposet of $\mathcal{F}(K)$ consisting of all faces of $C$ is isomorphic to the face poset of a product $\Delta^{d_1}\times\dots\times\Delta^{d_m}$ of simplices (with the usual faces); we say that $\Delta^{d_1}\times\dots\times\Delta^{d_m}$ is the \emph{shape} of $C$. (See \cite[Section~3]{Bjo84} for the definitions of regular CW complex and its face poset.)
Elements of the $0$-skeleton $K^{(0)}$ of a prodsimplicial complex $K$ are called \emph{vertices} of $K$.
A prodsimplicial complex $K$ is \emph{flag} if it has the following property, which implies that $K$ is completely determined by its $1$-skeleton $K^{(1)}$: for any product of simplices $\Delta^{d_1}\times\dots\times\Delta^{d_m}$, the cells $C$ in $K$ of this shape bijectively correspond to the induced subgraphs of $K^{(1)}$ isomorphic to the $1$-skeleton $(\Delta^{d_1}\times\dots\times\Delta^{d_m})^{(1)}$ of $\Delta^{d_1}\times\dots\times\Delta^{d_m}$, under the map $C\mapsto C^{(1)}$.
The Hom complex $\Hom(G,H)$ is a flag prodsimplicial complex \cite[Proposition~8.18]{Koz08}, and hence it is determined by the graph $\Hom(G,H)^{(1)}$ (described in the first paragraph of \cref{sec:Introduction}) in the above manner. 
Another way to define $\Hom(G,H)$ is to define its face poset $\mathcal{F}\bigl(\Hom(G,H)\bigr)$, since a regular CW complex can be reconstructed up to homeomorphism from its face poset \cite[Section~3]{Bjo84}. Although it is easy to describe $\mathcal{F}\bigl(\Hom(G,H)\bigr)$ explicitly (see, e.g.,~\cite[Definition~1.2]{BaKo06}), we will not use this point of view in this paper.
See \cite[Chapter~18]{Koz08} or \cite[Section~2]{BaKo06} for several examples of Hom complexes.

A \emph{cube complex} is a prodsimplicial complex $K$ such that, for each cell $C$ in $K$, the shape of $C$ is a hypercube $\Delta^1\times\dots \times \Delta^1$. Given a cube complex $K$ and a subset $S^{(0)}\subseteq K^{(0)}$ of its $0$-skeleton, the \emph{induced cube subcomplex} $S$ of $K$ determined by $S^{(0)}$ is the cube subcomplex of $K$ whose $0$-skeleton is $S^{(0)}$ and whose cells are all cells $C$ in $K$ such that $C\cap K^{(0)}\subseteq S^{(0)}$. (This notion makes sense more generally for arbitrary prodsimplicial complexes, for example, but we will use it only for cube complexes.)

Given a topological space $X$ and a point $a\in X$, we denote by $\Conn(X,a)$ the connected component of $X$ containing $a$.

For each nonnegative integer $k$, we have a canonical surjection $\Z\to \Z/k\Z$ mapping each integer $m$ to the element $[m]_k$ of $\Z/k\Z$ represented by $m$.
For $k\geq 3$, we identify the set of vertices of the $k$-cycle graph $C_k$ with $\Z/k\Z=\{[0]_k,[1]_k,\dots, [k-1]_k\}$, in such a way that $x,y\in \Z/k\Z$ are adjacent if and only if $y-x\in\{\pm [1]_k\}$. 
As in this last statement, we often perform arithmetical operations on the vertices of $C_k$, using the fact that $\Z/k\Z$ is an abelian group under addition. 

The crucial feature of the $k$-cycle $C_k$ with $k\neq 4$ is that the homomorphisms to $C_k$ have the \emph{monochromatic neighborhood property} (cf.~\cite{Wro20,LMS25}), mentioned in \cref{sec:Introduction}. Here is a more precise statement.
Recall that two homomorphisms $f,g\colon G\to H$ are adjacent if there exists $u\in V(G)$ such that $f(u)\neq g(u)$ and $f(v)=g(v)$ for all $v\in V(G)\setminus\{u\}$.

\begin{proposition}
\label{prop:MNP}
    Let $G$ be a graph and $f,g\colon G\to C_k$ be adjacent homomorphisms, where $k\geq 3$ and $k\neq 4$. Let $u\in V(G)$ be the vertex with $f(u)\neq g(u)$. Then, either of the following holds. 
    \begin{enumerate}[label=\emph{(\arabic*)}]
        \item For each vertex $v\in V(G)$ adjacent to $u$, we have $f(u)+[1]_k=f(v)=g(v)=g(u)-[1]_k$. 
        \item For each vertex $v\in V(G)$ adjacent to $u$, we have $f(u)-[1]_k=f(v)=g(v)=g(u)+[1]_k$. 
    \end{enumerate}
\end{proposition}
In particular, if $G$ has no isolated vertices, then each pair $(f,g)$ of adjacent homomorphisms $G\to C_k$ falls into precisely one of the cases (1) or (2) of \cref{prop:MNP}; we say that the (ordered) pair $(f,g)$ is of \emph{positive type} if (1) holds, and is of \emph{negative type} if (2) holds.

\cref{prop:MNP} gives us the following insight into the graph $\Hom(G,C_k)^{(1)}$, where $G$ is a graph without isolated vertices and $k\geq 3$ satisfies $k\neq 4$. Given any walk $W=(f_0,f_1,\dots,f_\ell)$ in $\Hom(G,C_k)^{(1)}$, we can assign to each vertex $u\in V(G)$ an integer $a^{(W)}_u$ defined as follows:
\begin{multline*}
    a^{(W)}_u=\bigl|\bigl\{\,i\in \{1,\dots, \ell\} \mid \text{$f_{i-1}(u)\neq f_i(u)$ and $(f_{i-1},f_i)$ is of positive type}\,\bigr\}\bigr|\\
    - \bigl|\bigl\{\,i\in \{1,\dots, \ell\} \mid \text{$f_{i-1}(u)\neq f_i(u)$ and $(f_{i-1},f_i)$ is of negative type}\,\bigr\}\bigr|.
\end{multline*}
Intuitively, $a^{(W)}_u$ denotes the (signed) number of times the value at $u$ has changed during $W$. 
It is easy to see that the element $\bm a^{(W)}=(a_v^{(W)})_{v\in V(G)}\in \Z^{V(G)}$ thus obtained from $W$ has to satisfy a few conditions, expressible in terms of the initial homomorphism $f_0$. For example, if we have $uv\in E(G)$ with $f_0(v)-f_0(u)=[1]_k$, then the inequalities $a_v^{(W)}\leq a_u^{(W)}\leq a_v^{(W)} +1$ have to hold. The following definitions are motivated by the above observation.

\begin{definition}
\label{def:G-L}
    Let $G$ be a graph, $k\geq 3$, and $f\colon G\to C_k$ a homomorphism.
    \begin{enumerate}[label={(\arabic*)}]
        \item Define the digraph $\overrightarrow{G}_f$ by $V(\overrightarrow{G}_f)=V(G)$ and 
\[
A(\overrightarrow{G}_f)=\{\,(u,v)\in V(G)^2 \mid \text{$uv\in E(G)$ and $f(v)-f(u)=[1]_k$}\,\}.
\]
We remark that $\overrightarrow{G}_f$ is obtained from the graph $G$ by choosing a direction of each edge. (That is, $\overrightarrow{G}_f$ is an \emph{orientation} of $G$ in the sense of \cite[Section~1.10]{Die17}.) 
See \cref{fig:G-f} for an example of $\overrightarrow{G}_f$.
        \item Define the subset $D_f^{(0)}\subseteq \Z^{V(G)}$ by
\begin{align*}
D_f^{(0)} &=\{\,(a_v)_{v\in V(G)}\in \Z^{V(G)}\mid \text{$a_v\leq a_u\leq a_v+1$ for each $(u,v)\in A(\overrightarrow{G}_f)$}\,\}.
\end{align*}
For $uv\in E(G)$ with $(u,v)\in A(\overrightarrow{G}_f)$, we call the inequalities $x_v\leq x_u\leq x_v+1$ the \emph{$uv$-th defining inequalities} of $D_f^{(0)}$.
    \item 
    There is a natural cube complex structure on $\R^{V(G)}$, in which $\Z^{V(G)}$ is the set of vertices (i.e., the $0$-skeleton) and whose maximal (closed) cubes are the $|V(G)|$-dimensional cubes 
\[
\{\,(c_v)_{v\in V(G)}\in \R^{V(G)}\mid \text{$a_v\leq c_v\leq a_v+1$ for each $v\in V(G)$}\,\}
= \prod_{v\in V(G)}[a_v,a_v+1]
\]
for each $(a_v)_{v\in V(G)}\in \Z^{V(G)}$.
    Define the cube complex $D_f$ as the induced cube subcomplex of $\R^{V(G)}$ determined by $D_f^{(0)}\subseteq \Z^{V(G)}$.
    Notice that the $1$-skeleton $D_f^{(1)}$ of $D_f$ is the (possibly infinite) graph whose vertex set is $D_f^{(0)}$, and in which two vertices $(a_v)_{v\in V(G)}$ and $(b_v)_{v\in V(G)}$ are adjacent if and only if $\sum_{v\in V(G)}|b_v-a_v|=1$. In other words, if we define $\bm e_u=(e_{u,v})_{v\in V(G)}\in \Z^{V(G)}$ by 
\[
    e_{u,v}=\begin{cases}
        1 &\text{if $u=v$, and}\\
        0 &\text{otherwise}
    \end{cases}
\]
for each $u\in V(G)$, then $\bm a,\bm b\in D_f^{(0)}$ are adjacent in  $D_f^{(1)}$ if and only if we have $\bm a=\bm b+\varepsilon\bm e_u$ for some $u\in V(G)$ and some $\varepsilon\in\{\pm 1\}$.
    \end{enumerate}
\end{definition}
\begin{figure}[h]
	\centering
	\begin{tikzpicture}[scale=0.2]
		\node(a01) at (0,3) {$\bullet$};
		\node(a02) at (0,9) {$\bullet$};
		\node(a03) at (4,0) {$\bullet$};
		\node(a04) at (4,6) {$\bullet$};
		\node(a05) at (4,12) {$\bullet$};
		\node(a06) at (10,0) {$\bullet$};
		\node(a07) at (10,6) {$\bullet$};
		\node(a08) at (10,12) {$\bullet$};
		\node(a09) at (14,2) {$\bullet$};
		\node(a10) at (14,10) {$\bullet$};
		\node(a11) at (16,6) {$\bullet$};
	  
		\draw (0,3) -- (0,9) -- (4,12) -- (10,12) -- (14,10) -- (16,6) -- (14,2) -- (10,0) -- (4,0) -- (0,3);
		\draw (4,0) -- (4,6) -- (4,12);
		\draw (10,0) -- (10,6) -- (10,12);
		\draw (4,6) -- (10,6);

		\node(b01) at (30,3) {$\bullet$};
		\node(b02) at (30,9) {$\bullet$};
		\node(b03) at (34,0) {$\bullet$};
		\node(b04) at (34,6) {$\bullet$};
		\node(b05) at (34,12) {$\bullet$};
		\node(b06) at (40,0) {$\bullet$};
		\node(b07) at (40,6) {$\bullet$};
		\node(b08) at (40,12) {$\bullet$};
		\node(b09) at (44,2) {$\bullet$};
		\node(b10) at (44,10) {$\bullet$};
		\node(b11) at (46,6) {$\bullet$};

		\draw[->, shorten >=2pt] (40,12) -- (34,12);
		\draw[->, shorten >=2pt] (34,12) -- (30,9);
		\draw[->, shorten >=2pt] (30,9) -- (30,3);
		\draw[->, shorten >=2pt] (30,3) -- (34,0);
		\draw[->, shorten >=2pt] (34,0) -- (40,0);
		\draw[->, shorten >=2pt] (40,0) -- (44,2);
		\draw[->, shorten >=2pt] (44,2) -- (46,6);
		\draw[->, shorten >=2pt] (40,12) -- (44,10);
		\draw[->, shorten >=2pt] (44,10) -- (46,6);
		\draw[->, shorten >=2pt] (34,6) -- (34,12);
		\draw[->, shorten >=2pt] (34,0) -- (34,6);
		\draw[->, shorten >=2pt] (40,6) -- (40,0);
		\draw[->, shorten >=2pt] (40,6) -- (34,6);
		\draw[->, shorten >=2pt] (40,6) -- (40,12);

		\node at (a01) [left] {$1$};
		\node at (a02) [left] {$0$};
		\node at (a03) [below] {$2$};
		\node at (a04) [left] {$3$};
		\node at (a05) [above] {$4$};
		\node at (a06) [below] {$3$};
		\node at (a07) [right] {$2$};
		\node at (a08) [above] {$3$};
		\node at (a09) [below right] {$4$};
		\node at (a10) [above right] {$4$};
		\node at (a11) [right] {$0$};

		\node at (b01) [left] {$1$};
		\node at (b02) [left] {$0$};
		\node at (b03) [below] {$2$};
		\node at (b04) [left] {$3$};
		\node at (b05) [above] {$4$};
		\node at (b06) [below] {$3$};
		\node at (b07) [right] {$2$};
		\node at (b08) [above] {$3$};
		\node at (b09) [below right] {$4$};
		\node at (b10) [above right] {$4$};
		\node at (b11) [right] {$0$};


	\end{tikzpicture}
	\caption{Left: a graph $G$ and a homomorphism $f\colon G\to C_5$, where $f$ is specified by labels on vertices.
    Right: the digraph
    $\overrightarrow{G}_f$.}
	\label{fig:G-f}
\end{figure}

\begin{example}
\label{ex:path}
Let $G$ be the graph 
\[
\begin{tikzpicture}[baseline=-\the\dimexpr\fontdimen22\textfont2\relax ]
      \node(0) at (0,0) {$\bullet$};
      \node(1) at (1,0) {$\bullet$};
      \node(2) at (2,0) {$\bullet$};
      \draw (0,0) to (2,0);
      \node at (0.north) {$u$}; 
      \node at (1.north) {$v$}; 
      \node at (2.north) {$w$}; 
\end{tikzpicture}
\]
and $k\geq 3$. Define a homomorphism $f\colon G\to C_k$ by $f(u)=[0]_k$, $f(v)=[1]_k$, and $f(w)=[2]_k$.
Then $D_f$ is a $2$-dimensional complex illustrated in \cref{fig:P3Ck} (irrespective of the value of $k$). \cref{fig:ConnHom} illustrates $\Conn\bigl(\Hom(G,C_k),f\bigr)$ for $3\leq k\leq 6$. Notice that the case where $k=4$ is markedly different. In particular, whereas there is a covering map $D_f\to \Conn\bigl(\Hom(G,C_k),f\bigr)$ when $k=3,5$, or $6$, there is no covering map $D_f\to \Conn\bigl(\Hom(G,C_4),f\bigr)$, as $\Conn\bigl(\Hom(G,C_4),f\bigr)$ is a $3$-dimensional cube.
Also note that the number of homomorphisms $G\to C_k$ is $4k$. Thus \cref{fig:ConnHom} shows that 
we have $\Conn\bigl(\Hom(G,C_k),f\bigr)=\Hom(G,C_k)$ when $k=3$ or $5$, whereas we have $\Conn\bigl(\Hom(G,C_k),f\bigr)\subsetneq\Hom(G,C_k)$ when $k=4$ or $6$. 
In fact, it is not difficult to see that, for this $G$, $\Hom(G,C_k)$ is connected when $k$ is odd and $\Hom(G,C_k)$ has two connected components when $k$ is even.
\end{example}

\begin{figure}[h]
    \centering
		\begin{tikzpicture}[scale=0.2]
			\coordinate(A) at (0,0);
			\coordinate(B) at (3,2);
			\coordinate(C) at (33,28);
			\coordinate(D) at (36,30);
			
			\coordinate(A1) at (12,10);
			\coordinate(A2) at (6,10);
			\coordinate(A3) at (6,4);
			\coordinate(A4) at (12,4);
			
			\coordinate(B1) at (21,18);
			\coordinate(B2) at (15,18);
			\coordinate(B3) at (15,12);
			\coordinate(B4) at (21,12);
			
			\coordinate(C1) at (30,26);
			\coordinate(C2) at (24,26);
			\coordinate(C3) at (24,20);
			\coordinate(C4) at (30,20);
			
			\draw[fill=gray!20] (A1) -- (A2) -- (A3) -- (A4) -- cycle;
			\draw[fill=gray!20] (B1) -- (B2) -- (B3) -- (B4) -- cycle;
			\draw[fill=gray!20] (C1) -- (C2) -- (C3) -- (C4) -- cycle;
			
			\draw (A) -- (B)[dashed];
			\draw (B) -- (A3);
			\draw (A1) -- (B3);
			\draw (B1) -- (C3);
			\draw (C1) -- (C);
			\draw (C) -- (D)[dashed];
			
			\node at (A1) [below right] {$(1,0,0)$};
			\node at (A2) [above left] {$(0,0,0)$};
			\node at (A3) [above left] {$(0,0,-1)$};
			\node at (A4) [below right] {$(1,0,-1)$};
			
			\node at (B1) [below right] {$(2,1,1)$};
			\node at (B2) [above left] {$(1,1,1)$};
			\node at (B3) [above left] {$(1,1,0)$};
			\node at (B4) [below right] {$(2,1,0)$};
			
			\node at (C1) [below right] {$(3,2,2)$};
			\node at (C2) [above left] {$(2,2,2)$};
			\node at (C3) [above left] {$(2,2,1)$};
			\node at (C4) [below right] {$(3,2,1)$};
			
			\draw[->] (30,0) -- (36,0) node[anchor=north] {$x_u$};
			\draw[->] (30,0) -- (33,2) node[anchor=south west] {$x_v$};
			\draw[->] (30,0) -- (30,6) node[anchor=east] {$x_w$};
		\end{tikzpicture}
    \caption{$D_f$ defined by $x_v\leq x_u\leq x_v+1$ and $x_w\leq x_v\leq x_w+1$ (see \cref{ex:path}).}
    \label{fig:P3Ck}
\end{figure}

\begin{figure}[h]
	\centering
    \begin{tikzpicture}[scale=0.19]
		
		\coordinate(A1) at (12,10);
		\coordinate(A2) at (6,10);
		\coordinate(A3) at (6,4);
		\coordinate(A4) at (12,4);
		
		\coordinate(B1) at (21,18);
		\coordinate(B2) at (15,18);
		\coordinate(B3) at (15,12);
		\coordinate(B4) at (21,12);
		
		\coordinate(C1) at (30,26);
		\coordinate(C2) at (24,26);
		\coordinate(C3) at (24,20);
		\coordinate(C4) at (30,20);
		
		\draw (A1) -- (A2) -- (A3) -- (A4) -- cycle;
		\draw (B1) -- (B2) -- (B3) -- (B4) -- cycle;
		\draw (C1) -- (C2) -- (C3) -- (C4) -- cycle;

		\draw (A1) -- (B3);
		\draw (B1) -- (C3);
		
		\node at (A1) [below right] {$212$};
		\node at (A2) [above left] {$012$};
		\node at (A3) [above left] {$010$};
		\node at (A4) [above right] {$210$};
		
		\node at (B1) [below right] {$101$};
		\node at (B2) [above left] {$201$};
		\node at (B3) [above left] {$202$};
		\node at (B4) [above right] {$102$};
		
		\node at (C1) [above right] {$020$};
		\node at (C2) [above left] {$120$};
		\node at (C3) [above left] {$121$};
		\node at (C4) [below left] {$021$};
        \draw[fill=gray!20] (A1) -- (A2) -- (A3) -- (A4) -- cycle;
			\draw[fill=gray!20] (B1) -- (B2) -- (B3) -- (B4) -- cycle;
			\draw[fill=gray!20] (C1) -- (C2) -- (C3) -- (C4) -- cycle;
		
		\draw (A3) .. controls (6,-10) and (45,26) .. (C1);

        \node at (8,24) {$k=3$};
    \begin{scope}[shift={(2,-22)}]
		\coordinate(A1) at (12,10);
		\coordinate(A2) at (4,10);
		\coordinate(A3) at (4,2);
		\coordinate(A4) at (12,2);
		
		\coordinate(B1) at (16,13);
		\coordinate(B2) at (8,13);
		\coordinate(B3) at (8,5);
		\coordinate(B4) at (16,5);
        \draw[fill=gray!20] (A3) -- (A4) -- (B4)-- (B1) -- (B2) -- (A2) -- cycle;
        
		\draw (A1) -- (A2) -- (A3) -- (A4) -- cycle;
		\draw (B1) -- (B2) -- (B3) -- (B4) -- cycle;
		\draw (A1) -- (B1);
		\draw (A2) -- (B2);
		\draw (A3) -- (B3);
		\draw (A4) -- (B4);

		\node at (A1) [below right] {$212$};
		\node at (A2) [above left] {$012$};
		\node at (A3) [below left] {$010$};
		\node at (A4) [below right] {$210$};
		
		\node at (B1) [above right] {$232$};
		\node at (B2) [above left] {$032$};
		\node at (B3) [above left] {$030$};
		\node at (B4) [above right] {$230$};

        \node at (4,-4) {$k=4$};
    \end{scope}
    \begin{scope}[shift={(15,-40)}]
    \clip (0,0) rectangle (53, 45);
		\coordinate(A1) at (12,10);
		\coordinate(A2) at (6,10);
		\coordinate(A3) at (6,4);
		\coordinate(A4) at (12,4);
		
		\coordinate(B1) at (21,18);
		\coordinate(B2) at (15,18);
		\coordinate(B3) at (15,12);
		\coordinate(B4) at (21,12);
		
		\coordinate(C1) at (30,26);
		\coordinate(C2) at (24,26);
		\coordinate(C3) at (24,20);
		\coordinate(C4) at (30,20);
		
		\coordinate(D1) at (39,34);
		\coordinate(D2) at (33,34);
		\coordinate(D3) at (33,28);
		\coordinate(D4) at (39,28);
		
		\coordinate(E1) at (48,42);
		\coordinate(E2) at (42,42);
		\coordinate(E3) at (42,36);
		\coordinate(E4) at (48,36);
		
		\draw (A1) -- (A2) -- (A3) -- (A4) -- cycle;
		\draw (B1) -- (B2) -- (B3) -- (B4) -- cycle;
		\draw (C1) -- (C2) -- (C3) -- (C4) -- cycle;
		\draw (D1) -- (D2) -- (D3) -- (D4) -- cycle;
		\draw (E1) -- (E2) -- (E3) -- (E4) -- cycle;

		\draw (A1) -- (B3);
		\draw (B1) -- (C3);
		\draw (C1) -- (D3);
		\draw (D1) -- (E3);
		
		\node at (A1) [below right] {$212$};
		\node at (A2) [above left] {$012$};
		\node at (A3) [above left] {$010$};
		\node at (A4) [above right] {$210$};
		
		\node at (B1) [below right] {$434$};
		\node at (B2) [above left] {$234$};
		\node at (B3) [above left] {$232$};
		\node at (B4) [above right] {$432$};
		
		\node at (C1) [below right] {$101$};
		\node at (C2) [above left] {$401$};
		\node at (C3) [above left] {$404$};
		\node at (C4) [below left] {$104$};
		
		\node at (D1) [below right] {$323$};
		\node at (D2) [above left] {$123$};
		\node at (D3) [above left] {$121$};
		\node at (D4) [below left] {$321$};
		
		\node at (E1) [above right] {$040$};
		\node at (E2) [above left] {$340$};
		\node at (E3) [above left] {$343$};
		\node at (E4) [below left] {$043$};
		
		\draw (A3) .. controls (6,-13) and (66,42) .. (E1);
            \draw[fill=gray!20] (A1) -- (A2) -- (A3) -- (A4) -- cycle;
			\draw[fill=gray!20] (B1) -- (B2) -- (B3) -- (B4) -- cycle;
			\draw[fill=gray!20] (C1) -- (C2) -- (C3) -- (C4) -- cycle;
        \draw[fill=gray!20] (D1) -- (D2) -- (D3) -- (D4) -- cycle;
			\draw[fill=gray!20] (E1) -- (E2) -- (E3) -- (E4) -- cycle;

        \node at (40,14) {$k=5$};
        
    \end{scope}
    \begin{scope}[shift={(30,0)}]
		\clip (0,0) rectangle (35, 30);
		\coordinate(A1) at (12,10);
		\coordinate(A2) at (6,10);
		\coordinate(A3) at (6,4);
		\coordinate(A4) at (12,4);
		
		\coordinate(B1) at (21,18);
		\coordinate(B2) at (15,18);
		\coordinate(B3) at (15,12);
		\coordinate(B4) at (21,12);
		
		\coordinate(C1) at (30,26);
		\coordinate(C2) at (24,26);
		\coordinate(C3) at (24,20);
		\coordinate(C4) at (30,20);
		
		\draw (A1) -- (A2) -- (A3) -- (A4) -- cycle;
		\draw (B1) -- (B2) -- (B3) -- (B4) -- cycle;
		\draw (C1) -- (C2) -- (C3) -- (C4) -- cycle;

		\draw (A1) -- (B3);
		\draw (B1) -- (C3);
		
		\node at (A1) [below right] {$212$};
		\node at (A2) [above left] {$012$};
		\node at (A3) [above left] {$010$};
		\node at (A4) [above right] {$210$};
		
		\node at (B1) [below right] {$434$};
		\node at (B2) [above left] {$234$};
		\node at (B3) [above left] {$232$};
		\node at (B4) [above right] {$432$};
		
		\node at (C1) [above right] {$050$};
		\node at (C2) [above left] {$450$};
		\node at (C3) [above left] {$454$};
		\node at (C4) [below left] {$054$};
		
		\draw (A3) .. controls (6,-10) and (45,26) .. (C1);
        \draw[fill=gray!20] (A1) -- (A2) -- (A3) -- (A4) -- cycle;
			\draw[fill=gray!20] (B1) -- (B2) -- (B3) -- (B4) -- cycle;
			\draw[fill=gray!20] (C1) -- (C2) -- (C3) -- (C4) -- cycle;
        \node at (12,24) {$k=6$};
     \end{scope}
    \end{tikzpicture}
	\caption{$\Conn\bigl(\Hom(G,C_k),f\bigr)$ for $k=3$, $4$, $5$, and $6$ (counterclockwise from top left; see \cref{ex:path}). The label $ij\ell$ indicates the homomorphism $g\colon G\to C_k$ defined by $g(u)=[i]_k$, $g(v)=[j]_k$, and $g(w)=[\ell]_k$.}
	\label{fig:ConnHom}
\end{figure}

Let $G$ be a graph, $k\geq 3$ an integer, and $f\colon G\to C_k$ a graph homomorphism.
As we shall show in \cref{sec:Construction}, there is a natural continuous map $p_f\colon D_f\to \Hom(G,C_k)$ and, under suitable conditions, we can obtain the universal cover of the connected component $\Conn\bigl(\Hom(G,C_k),f\bigr)$ by restricting $p_f$. 
In fact, we can already define the action of $p_f$ on vertices.


\begin{proposition}
\label{prop:p-well-defined}
    Let $G$ be a graph, $k\geq 3$, $f\colon G\to C_k$ a homomorphism, and  $\bm a=(a_v)_{v\in V(G)}\in D_f^{(0)}$. Then the function $p_f(\bm a)\colon V(G)\to \Z/k\Z$ defined by 
    \[
    p_f(\bm a)(u)=f(u)+[2a_u]_k
    \]
    for each $u\in V(G)$, is a homomorphism $p_f(\bm a)\colon G\to C_k$. More precisely, for any edge $uv\in E(G)$ with $(u,v)\in A(\overrightarrow{G}_f)$, we have 
    \begin{equation}
    \label{eqn:p-alpha-a-cases}
        p_f(\bm a)(v)-p_f(\bm a)(u)=\begin{cases}
            [1]_k  &\text{if $a_u=a_v$, and}\\
            -[1]_k &\text{if $a_u=a_v+1$.}
        \end{cases}
    \end{equation}
\end{proposition}
\begin{proof}
    The first assertion follows from the second, which is straightforward to check.
\end{proof}

\cref{prop:p-well-defined} shows that $\bm a\mapsto p_f(\bm a)$ defines a function $p_f^{(0)}\colon D_f^{(0)}\to \Hom(G,C_k)^{(0)}$, where $ \Hom(G,C_k)^{(0)}$ is the set of all homomorphisms $G\to C_k$. 
It is illustrative to see the action of this function in the situation of \cref{ex:path}; in particular, \cref{fig:P3Ck,fig:ConnHom} show that it defines the action on vertices of a covering map $D_f\to \Conn\bigl(\Hom(G,C_k),f\bigr)$ when $k=3,5$, and $6$.
We put the superscript $(0)$ on the names of the set $D_f^{(0)}$ and of the function $p_f^{(0)}$, because $p_f^{(0)}\colon D_f^{(0)}\to \Hom(G,C_k)^{(0)}$ will turn out to be the restriction to the $0$-skeletons of a map $p_f\colon D_f\to\Hom(G,C_k)$ of cube complexes. However, we avoid putting the superscript when denoting a value $p_f(\bm a)$ of the function $p_f^{(0)}$. 
We adopt similar notational conventions in what follows. 

The following proposition is obvious; its proof is omitted.
\begin{proposition}
\label{prop:p-hom-of-graphs}
    Let $G$ be a graph, $k\geq 3$, and $f\colon G\to C_k$ a graph homomorphism.
    Then the function $p_f^{(0)}\colon D_f^{(0)}\to \Hom(G,C_k)^{(0)}$ defines a graph homomorphism $p_f^{(1)}\colon D_f^{(1)}\to \Hom(G,C_k)^{(1)}$.
\end{proposition}

The following proposition shows that there is a close relationship between the notions just introduced, provided that $k\neq 4$. It is used repeatedly throughout this paper.

\begin{proposition}
\label{prop:local-str-of-L}
    Let $G$ be a graph without isolated vertices, $k$ an integer with $k\geq 3$ and $k\neq 4$, and $f\colon G\to C_k$ a graph homomorphism.
    Let $\bm a\in D^{(0)}_f$ and $u\in V(G)$. Then the following conditions are equivalent. 
       \begin{enumerate}[label=\emph{(\arabic*)}]
        \item The vertex $u$ is a source in $\overrightarrow{G}_{p_f(\bm a)}$. 
        \item We have $\bm a+\bm e_u\in D^{(0)}_f$.
        \item There exists a homomorphism $g\colon G\to C_k$ adjacent to $p_f(\bm a)$ such that $p_f(\bm a)(u)\neq g(u)$ and $\bigl(p_f(\bm a),g\bigr)$ is of positive type. 
    \end{enumerate}
    Similarly, the following conditions are equivalent. 
    \begin{enumerate}[label=\emph{(\arabic*)}]
    \setcounter{enumi}{3}
        \item The vertex $u$ is a sink in $\overrightarrow{G}_{p_f(\bm a)}$. 
        \item We have $\bm a-\bm e_u\in D^{(0)}_f$.
        \item There exists a homomorphism $g\colon G\to C_k$ adjacent to $p_f(\bm a)$ such that $p_f(\bm a)(u)\neq g(u)$ and $\bigl(p_f(\bm a),g\bigr)$ is of negative type. 
    \end{enumerate}
\end{proposition}

Before proceeding to a proof of \cref{prop:local-str-of-L}, we make a few comments. 
First, the assumption that $G$ should not have any isolated vertices is not explicitly used in the following proof, but is included in order to ensure that the notions of a positive or negative type (introduced just after \cref{prop:MNP}) are well defined.
Second, it may be helpful to look at \cref{fig:G-f} to gain some intuition for \cref{prop:local-str-of-L}.
For example, let $u$ be the middle vertex labeled $2$ in \cref{fig:G-f}. Then, $u$ is a source in $\overrightarrow{G}_f$, i.e., (1) of \cref{prop:local-str-of-L} is satisfied (with $\mathbf{a}=\mathbf{0}=(0)_{v\in V(G)}$). 
On the other hand, we can change the value of $f$ at $u$ from $[2]_5$ to $[4]_5$ to obtain a new graph homomorphism $g\colon G\to C_k$ such that $(f,g)$ is of positive type, i.e., (3) of \cref{prop:local-str-of-L} is satisfied (again with $\mathbf{a}=\mathbf{0}$). 

\begin{proof}[Proof of \cref{prop:local-str-of-L}]
    We show the equivalence of (1), (2), and (3). 
    First observe that a homomorphism $g$ satisfying the condition of (3) is necessarily given by 
\begin{equation}
\label{eqn:g}
        g(v)=\begin{cases}
            p_f(\bm a)(u)+[2]_k &\text{if $v=u$, and}\\
            p_f(\bm a)(v)   &\text{if $v\neq u$}
        \end{cases}
\end{equation}
    for each $v\in V(G)$. Therefore, (3) is equivalent to the condition that the function $g\colon V(G)\to \Z/k\Z$ defined by \cref{eqn:g} should be a homomorphism $g\colon G\to C_k$.

    Now consider the following statements for a vertex $v\in V(G)$ adjacent to $u$.
\begin{itemize}
    \item[($1'$)] We have $(u,v)\in A(\overrightarrow{G}_{p_f(\bm a)})$.
    \item[($2'$)] The element $\bm a + \bm e_u\in \Z^{V(G)}$ satisfies the $uv$-th defining inequalities of $D_f^{(0)}$. 
    \item[($3'$)] The function $g$ defined by \cref{eqn:g} satisfies $g(v)-g(u) = f(v)-f(u) + [2a_v -2a_u-2]_k\in \{\pm [1]_k\}$.
\end{itemize}
     It is easy to see that, for each $i\in \{1,2,3\}$, ($i$) is equivalent to the condition that ($i'$) should be satisfied by all vertices $v\in V(G)$ adjacent to $u$.
     Therefore it suffices to show the equivalence of ($1'$), ($2'$), and ($3'$), for each vertex $v\in V(G)$ adjacent to $u$.
    
    Let us consider the case where $(u,v)\in A(\overrightarrow{G}_f)$; the case where $(v,u)\in A(\overrightarrow{G}_f)$ is similar. The $uv$-th defining inequalities of $D^{(0)}_f$ are $x_v\leq x_u\leq x_v+1$, and hence we have 
    \begin{equation}
    \label{eqn:uv-def-ineq-a}
         a_v\leq a_u\leq a_v+1
    \end{equation}
    by the assumption $\bm a\in D_f^{(0)}$. We have 
    \begin{align*}
        (1') &\overset{\cref{eqn:p-alpha-a-cases}}{\iff} a_u=a_v, \\
       (2') &\iff a_v\leq a_u+1\leq a_v+1, \text{ and}\\
       (3') &\iff [2a_v-2a_u-1]_k\in \{\pm [1]_k\}.
    \end{align*}
    Clearly ($1'$) implies ($2'$), and ($2'$) implies ($3'$). Given \cref{eqn:uv-def-ineq-a}, either $a_u=a_v$ or $a_u=a_v+1$ holds. In the former case, we have $[2a_v-2a_u-1]_k=[-1]_k$. In the latter case, we have $[2a_v-2a_u-1]_k=[-3]_k$, which is neither $[1]_k$ nor $[-1]_k$ since $k\geq 3$ and $k\neq 4$. Hence for ($3'$) to be satisfied, we have to have $a_u=a_v$; that is, ($3'$) implies ($1'$). 
    This completes the proof of the equivalence of (1), (2), and (3).
    In much the same way, one can show the equivalence of (4), (5), and (6). 
\end{proof}

Our aim in this paper is to determine the homotopy type of each connected component of $\Hom(G,C_k)$, for any graph $G$ and all $k\geq 3$. 
Observe that we can restrict our attention to the case where $G$ is connected, since in general we have $\Hom(G\sqcup G',H)\cong \Hom(G,H)\times \Hom(G',H)$ \cite[Section~2.4]{BaKo06},
where $G\sqcup G'$ denotes the disjoint union of two graphs $G$ and $G'$.
We can easily treat the cases where $k=4$ or where $G$ consists of a single isolated vertex.
\begin{proposition}
\label{prop:square}
    Let $G$ be a graph. Then each connected component of $\Hom(G,C_4)$ is contractible.
\end{proposition}
\begin{proof}
    Applying the folding theorem \cite[Theorem~18.22]{Koz08} twice, we see that $\Hom(G,C_4)$ is homotopy equivalent to $\Hom(G,K_2)$. It is easy to see that $\Hom(G,K_2)$ is homotopy equivalent to a disjoint union of points. For example, when $G$ is connected (and nonempty), then $\Hom(G,K_2)$ is empty if $G$ is non-bipartite, is the disjoint union of two points if $G$ is bipartite and has at least two vertices, and is a $1$-simplex if $G$ consists of a single vertex.
\end{proof}
\begin{proposition}
\label{prop:isolated}
    Let $G$ be a graph with a single vertex and $k\geq 3$. Then $\Hom(G,C_k)$ is contractible.
\end{proposition}
\begin{proof}
    In this case, $\Hom(G,C_k)$ (or more generally, $\Hom(G,H)$ where $H$ is any graph with $k$ vertices) is a $(k-1)$-simplex.
\end{proof}

These observations allow us to focus on the case where $G$ is a connected graph with at least two vertices and where the integer $k\geq 3$ satisfies $k\neq 4$, in what follows.

\section{Construction of the universal cover}
\label{sec:Construction}
For the rest of this paper, we fix a connected graph $G$ with at least two vertices, an integer $k$ with $k\geq 3$ and $k\neq 4$, and a homomorphism $f\colon G\to C_k$.
Notice that the first condition implies that $G$ has no isolated vertices, and hence the notions of a positive or negative type (defined just after \cref{prop:MNP}) are well defined. Thus we can use these notions, as well as \cref{prop:local-str-of-L}, in what follows.

Our aim is to determine the homotopy type of $\Conn\bigl(\Hom(G,C_k),f\bigr)$.
It will turn out that the existence or absence of directed cycles in the digraph $\overrightarrow{G}_f$ defined in \cref{def:G-L} (1) determines the homotopy type of $\Conn\bigl(\Hom(G,C_k),f\bigr)$; see \cref{thm:homotopy-type-of-Hom}.

In this section, we construct the universal cover of $\Conn\bigl(\Hom(G,C_k),f\bigr)$.
Here is an outline of the construction.
\begin{itemize}
    \item We take a suitable quotient $D_f/k'\Z$ of the cube complex $D_f$ defined in \cref{def:G-L} (3), and observe that the quotient map $q_f\colon D_f\to D_f/k'\Z$ is a covering map. 
    \item We observe that the quotient $D_f/k'\Z$ has a natural cube complex structure, and that 
    the graph homomorphism $p^{(1)}_f\colon D^{(1)}_f\to \Hom(G,C_k)^{(1)}$ (obtained in \cref{prop:p-hom-of-graphs}) induces a graph homomorphism $r^{(1)}_f\colon (D_f/k'\Z)^{(1)}\to \Hom(G,C_k)^{(1)}$ satisfying $p^{(1)}_f=r^{(1)}_f\circ q^{(1)}_f$.
    \item We show that the above graph homomorphism $r^{(1)}_f$ is an isomorphism onto a union of connected components in $\Hom(G,C_k)^{(1)}$ (\cref{prop:r-iso-onto-union-of-conn-comp}). 
    \item We observe that both $\Hom(G,C_k)$ and $D_f/k'\Z$ are \emph{flag} prodsimplicial complexes (\cref{prop:L-quotient-flag}). This, together with the previous fact (\cref{prop:r-iso-onto-union-of-conn-comp}), implies that the graph homomorphism $r^{(1)}_f$ induces a continuous map $r_f\colon D_f/k'\Z\to \Hom(G,C_k)$, which is a homeomorphism onto a union of connected components in $\Hom(G,C_k)$.
    \item We define a continuous map $p_f\colon D_f\to \Hom(G,C_k)$ as the composite $p_f=r_f\circ q_f$.
    Restricting it to suitable connected components, we obtain a covering map $p_f\colon \Conn(D_f,\mathbf{0})\to \Conn\bigl(\Hom(G,H),f\bigr)$ (\cref{prop:p-covering}), which turns out to be the universal cover of $\Conn\bigl(\Hom(G,H),f\bigr)$ (as shown in \cref{sec:Contractibility}).
\end{itemize} 



Let us start with the definition of the quotient $D_f/k'\Z$.
We have an action of the additive group $\Z$ of integers on $\R^{V(G)}$, defined by 
\[
m\cdot (c_v)_{v\in V(G)}= (m+c_v)_{v\in V(G)}
\]
for each $m\in\Z$ and $(c_v)_{v\in V(G)}\in \R^{V(G)}$. Since this action preserves the cube complex structure of $\R^{V(G)}$ and leaves $D_f^{(0)}$ invariant, it restricts to a $\Z$-action on the cube complex $D_f$. Notice that we have a bijection $m\cdot (-)\colon D_f^{(0)}\to D_f^{(0)}$ as well as a graph automorphism $m\cdot (-)\colon D_f^{(1)}\to D_f^{(1)}$ for each $m\in\Z$.
Also note that this $\Z$-action on $D_f$ is a covering space action in the sense of \cite[Section~1.3]{Hat02}, and hence for each subgroup $m\Z$ of $\Z$ (where $m$ is a nonnegative integer), the quotient map $D_f\to D_f/m\Z$ is a covering map.

Let us now define the positive integer $k'$ by $k'=k$ if $k$ is odd, and $k'=k/2$ if $k$ is even. 
Notice that $k'$ is the smallest positive integer $m$ making the triangle
\begin{equation}
\label{eqn:comm-triangle-k'}
\begin{tikzpicture}[baseline=-\the\dimexpr\fontdimen22\textfont2\relax]
      \node(00) at (0,1) {$D_f^{(1)}$};
      \node(01) at (4,1) {$D_f^{(1)}$};
      \node(10) at (2,-1) {$\Hom(G,C_k)^{(1)}$};
      \draw [->] (00) to node[auto, labelsize] {$m\cdot(-)$} (01);
      \draw [->] (01) to node[auto, labelsize] {$p_f^{(1)}$} (10);
      \draw [->] (00) to node[auto,swap,labelsize] {$p_f^{(1)}$} (10);
\end{tikzpicture}
\end{equation}
commutative.
In other words, $k'$ is the smallest positive integer $m$ such that, after $m$ moves of length $2$ in a fixed direction in $C_k$, one returns to the original vertex.
See \cref{fig:kprime} for an illustration. Observe that we have $k'\geq 3$ since $k \neq 4$.
\begin{figure}[h]
    \centering
\begin{equation*}
\begin{tikzpicture}
\node[draw,minimum size=2cm,regular polygon,regular polygon sides=5] (a) {};
\foreach \x in {1,2,...,5}
  \fill (a.corner \x) circle[radius=2pt];
  \begin{scope}[decoration={
    markings,
    mark=at position 0.15 with {\arrow{<}}}
    ] 
  \draw[dashed,postaction={decorate}] (a.corner 1) to (a.corner 4);
  \draw[dashed,postaction={decorate}] (a.corner 4) to (a.corner 2);
  \draw[dashed,postaction={decorate}] (a.corner 2) to (a.corner 5);
  \draw[dashed,postaction={decorate}] (a.corner 5) to (a.corner 3);
  \draw[dashed,postaction={decorate}] (a.corner 3) to (a.corner 1);
  \end{scope}
\end{tikzpicture}\qquad\qquad
\begin{tikzpicture}
\node[draw,minimum size=2cm,regular polygon,regular polygon sides=6] (a) {};
\foreach \x in {1,2,...,6}
  \fill (a.corner \x) circle[radius=2pt];
  \begin{scope}[decoration={
    markings,
    mark=at position 0.16 with {\arrow{<}}}
    ] 
  \draw[dashed,postaction={decorate}] (a.corner 1) to (a.corner 5);
  \draw[dashed,postaction={decorate}] (a.corner 5) to (a.corner 3);
  \draw[dashed,postaction={decorate}] (a.corner 3) to (a.corner 1);
  \end{scope}
\end{tikzpicture}
\end{equation*}
    \caption{Cycles with dashed arrows representing moves of length $2$ in the cases $k=5$ with $k'=5$ (left) and $k=6$ with $k'=3$ (right).}
    \label{fig:kprime}
\end{figure}

We take the quotient $q_f\colon D_f\to D_f/k'\Z$ (cf.~\cref{fig:P3Ck,fig:ConnHom}). Since $k'\geq 3$, the quotient space $D_f/k'\Z$ has the cube complex structure so that $q_f$ becomes a map of cube complexes. 
Therefore we also have the induced surjective graph homomorphism $q_f^{(1)}\colon D_f^{(1)}\to (D_f/k'\Z)^{(1)}$.
Here, the graph $(D_f/k'\Z)^{(1)}$ is the $1$-skeleton of the cube complex $D_f/k'\Z$, and can be described explicitly as follows. 
Its set of vertices is the quotient set $D_f^{(0)}/k'\Z$ consisting of equivalence classes $[\bm a]_{k'}= \{\, mk'\cdot \bm a \mid m\in\Z\,\}$, where $\bm a\in D_f^{(0)}$.
Two vertices $[\bm a]_{k'}$ and $[\bm b]_{k'}$ of $(D_f/k'\Z)^{(1)}$ are adjacent in $(D_f/k'\Z)^{(1)}$ if and only if there exist $\bm a'\in [\bm a]_{k'}$ and $\bm b'\in [\bm b]_{k'}$ such that $\bm a'$ and $\bm b'$ are adjacent in $D^{(1)}_f$. In fact, the latter condition is equivalent to the following seemingly stronger condition: \emph{for each} $\bm a'\in [\bm a]_{k'}$, there exists a \emph{unique} $\bm b'\in [\bm b]_{k'}$ such that $\bm a'$ and $\bm b'$ are adjacent in $D^{(1)}_f$.
It follows that the graph homomorphism $q_f^{(1)}\colon D_f^{(1)}\to (D_f/k'\Z)^{(1)}$, which maps a vertex $\bm a$ of $D_f^{(1)}$ to $q_f(\bm a)=[\bm a]_{k'}$, has the unique edge-lifting property. 

Since \cref{eqn:comm-triangle-k'} commutes with $m=k'$ (and since $(D_f/k'\Z)^{(1)}$ agrees with the quotient graph $D_f^{(1)}/k'\Z$ in a suitable sense), there exists a unique graph homomorphism $r_f^{(1)}\colon (D_f/k'\Z)^{(1)}\to\Hom(G,C_k)^{(1)}$ making the diagram  
\begin{equation*}
\begin{tikzpicture}[baseline=-\the\dimexpr\fontdimen22\textfont2\relax ]
      \node(00) at (0,1) {$D_f^{(1)}$};
      \node(01) at (3,1) {$(D_f/k'\Z)^{(1)}$};
      \node(10) at (3,-1) {$\Hom(G,C_k)^{(1)}$};
      
      \draw [->] (00) to node[auto, labelsize] {$q_f^{(1)}$} (01); 
      \draw [->] (01) to node[auto, labelsize] {$r_f^{(1)}$} (10); 
      \draw [->] (00) to node[auto,swap,labelsize] {$p_f^{(1)}$} (10); 
\end{tikzpicture}
\end{equation*}
commutative. Explicitly, $r_f^{(1)}$ maps a vertex $[\bm a]_{k'}$ of $(D_f/k'\Z)^{(1)}$ (with $\bm a\in D_f^{(0)}$) to $r_f([\bm a]_{k'})=p_f(\bm a)\in \Hom(G,C_k)^{(0)}$, which is well defined thanks to the commutativity of \cref{eqn:comm-triangle-k'} with $m=k'$.

\begin{proposition}
\label{prop:r-iso-onto-union-of-conn-comp}
    The graph homomorphism $r_f^{(1)}\colon (D_f/k'\Z)^{(1)}\to\Hom(G,C_k)^{(1)}$ is an isomorphism onto a union of connected components in $\Hom(G,C_k)^{(1)}$.
\end{proposition}
\begin{proof}
    It suffices to show that $r_f^{(1)}$ is injective and has the (unique) edge-lifting property.
    
    First we show the injectivity of $r_f^{(1)}$. Let us take $\bm a = (a_v)_{v\in V(G)}, \bm b = (b_v)_{v\in V(G)}\in D_f^{(0)}$ such that $p_f(\bm a) = p_f(\bm b)$. It suffices to show $q_f(\bm a) = q_f(\bm b)$, i.e., that there exists $m\in\Z$ with $mk'\cdot \bm a=\bm b$. For any vertex $v\in V(G)$, since we have   $f(v)+[2a_v]_k=p_f(\bm a)(v) = p_f(\bm b)(v)=f(v)+[2b_v]_k$ in $\Z/k\Z$, there exists a (necessarily unique) integer $m_v$ with $m_vk'+a_v=b_v$ in $\Z$. We claim that $uv\in E(G)$ implies $m_u=m_v$. Indeed, we may assume $(u,v)\in A(\overrightarrow{G}_f)$ without loss of generality. Then we have $a_v\leq a_u\leq a_v+1$ and $b_v\leq b_u\leq b_v+1$, which imply $m_u=m_v$ since $k'\geq 3$. It follows that the integer $m_v$ does not depend on $v$, as $G$ is connected.

    The edge-lifting property of $r_f^{(1)}$ follows from that of $p_f^{(1)}$; we show the latter. Take any $\bm a\in D_f^{(0)}$ and $g\in \Hom(G,C_k)^{(0)}$ such that $p_f(\bm a)$ and $g$ are adjacent in $\Hom(G,C_k)^{(1)}$. We want to show that there exists $\bm b\in D_f^{(0)}$ such that $\bm a$ and $\bm b$ are adjacent in $D_f^{(1)}$ and satisfies $p_f(\bm b)=g$. (Incidentally, this $\bm b$ turns out to be unique, so $p_f^{(1)}$ has the \emph{unique} edge-lifting property.)
    Let $u\in V(G)$ be the vertex of $G$ such that $p_f(\bm a)(u)\neq g(u)$. 
    Now define 
    \[
    \bm b=\begin{cases}
        \bm a + \bm e_u &\text{if $\bigl(p_f(\bm a),g\bigr)$ is of positive type, and}\\
        \bm a - \bm e_u &\text{if $\bigl(p_f(\bm a),g\bigr)$ is of negative type.}
    \end{cases}
    \]
    Then we have $\bm b\in D_f^{(0)}$ by \cref{prop:local-str-of-L}. 
    Clearly $\bm a$ and $\bm b$ are adjacent in $D_f^{(1)}$ and we have $p_f(\bm b)=g$.
\end{proof}

Our next aim is to show that the graph homomorphism $r_f^{(1)}\colon (D_f/k'\Z)^{(1)}\to \Hom(G,C_k)^{(1)}$ extends to a map $r_f\colon D_f/k'\Z\to \Hom(G,C_k)$ of cube complexes, which is a homeomorphism onto a union of connected components in $\Hom(G,C_k)$. This would follow from \cref{prop:r-iso-onto-union-of-conn-comp}, once we know that both cube complexes $D_f/k'\Z$ and $\Hom(G,C_k)$ are flag prodsimplicial complexes. Since $\Hom(G,C_k)$ is a flag prodsimplicial complex by \cite[Proposition~18.1]{Koz08}, it suffices to show that $D_f/k'\Z$ is a flag prodsimplicial complex.

\begin{lemma}
\label{lem:square-in-L-quotient}
    Let $x_0,x_1,x_2,x_3\in (D_f/k'\Z)^{(0)}$ be vertices such that the induced subgraph of $(D_f/k'\Z)^{(1)}$ determined by $\{x_0,x_1,x_2,x_3\}$ is isomorphic to the square $C_4$, such that $x_i$ and $x_j$ are adjacent if and only if $[i-j]_4\in\{\pm [1]_4\}$. 
    Then, for any $\bm a\in D_f^{(0)}$ with $x_0=[\bm a]_{k'}$, there exist unique $u,v\in V(G)$ and unique $\delta,\varepsilon\in \{\pm 1\}$ such that $x_1=[\bm a+\delta\bm e_u]_{k'}$, $x_2=[\bm a+\delta\bm e_u+\varepsilon\bm e_v]_{k'}$, and $x_3=[\bm a+\varepsilon\bm e_v]_{k'}$. Moreover, we have $u\neq v$. 
\end{lemma}
\begin{proof}
    Assume that we have $x_0=[\bm a]_{k'}$. Then, we can write $x_1=[\bm a+\delta\bm e_u]_{k'}$ and $x_3=[\bm a+\varepsilon\bm e_v]_{k'}$ for some $u,v\in V(G)$ and some $\delta,\varepsilon\in \{\pm 1\}$. The uniqueness of $u,v,\delta,\varepsilon$ follows from $k'\geq 3$. If $u=v$, then we have $\{x_1,x_3\}= \{[\bm a-\bm e_u]_{k'}, [\bm a+\bm e_u]_{k'}\}$.
    However, this is impossible because by \cref{prop:local-str-of-L} we cannot have $\bm a-\bm e_u,\bm a + \bm e_u\in D_f^{(0)}$; note that a vertex in $G$ cannot be both a source and a sink in $\overrightarrow{G}_{p_f(\bm a)}$ since $G$ has no isolated vertices.
    Therefore we have $u\neq v$. 
    This forces $x_2$ to be $[\bm a+\varepsilon\bm e_u +\delta\bm e_v]_{k'}$.
\end{proof}

\begin{proposition}
\label{prop:L-quotient-flag}
    The cube complex $D_f/k'\Z$ is a flag prodsimplicial complex.
\end{proposition}
\begin{proof}
    Let $H$ be an induced subgraph of $(D_f/k'\Z)^{(1)}$ isomorphic to the $1$-skeleton of a product of simplices. First observe that $(D_f/k'\Z)^{(1)}$ does not contain a triangle $K_3$ since $k'\geq 3$ and $|V(G)|\geq 2$. Therefore $H$ must be the $1$-skeleton of a $d$-cube for some $d\geq 0$. It suffices to show that there exists a $d$-cube in $D_f$ whose vertices are mapped to the vertices of $H$ by $q_f\colon D_f\to D_f/k'\Z$.
    
    Take $\bm a\in D_f^{(0)}$ such that $[\bm a]_{k'}$ is a vertex of $H$. The $d$ vertices of $H$ adjacent to $[\bm a]_{k'}$ can be expressed as $[\bm a+ \varepsilon_1 \bm e_{v_1}]_{k'}, \dots, [\bm a+ \varepsilon_d \bm e_{v_d}]_{k'}$ for some $\varepsilon_1,\dots, \varepsilon_d\in \{\pm 1\}$ and some pairwise distinct $v_1,\dots,v_d\in V(G)$ by \cref{lem:square-in-L-quotient}. Applying \cref{lem:square-in-L-quotient} repeatedly, we see that the set of $2^d$ vertices of $H$ can be written as 
    $\{\,[\bm{a}+\delta_{1}\varepsilon_{1}\bm{e}_{v_1}+\dots+\delta_{d}\varepsilon_{d}\bm{e}_{v_d}]_{k'}\mid \delta_{i}\in \{0,1\}\text{ for all }1\leq i\leq d\,\}$.
    Since $\{\,\bm{a}+\delta_{1}\varepsilon_{1}\bm{e}_{v_1}+\dots+\delta_{d}\varepsilon_{d}\bm{e}_{v_d}\mid \delta_{i}\in \{0,1\}\text{ for all }1\leq i\leq d\,\}$ is the set of vertices of a $d$-cube in $D_f$, we see that $D_f/k'\Z$ contains a $d$-cube having $H$ as its $1$-skeleton. 
\end{proof}

Thus we obtain a continuous map $r_f\colon D_f/k'\Z\to \Hom(G,C_k)$. 
Define the continuous map $p_f\colon D_f\to \Hom(G,C_k)$ as the composite of $D_f\xrightarrow{q_f}D_f/k'\Z\xrightarrow{r_f}\Hom(G,C_k)$.
We consider its restriction $p_f\colon \Conn(D_f,\bm 0)\to \Conn\bigl(\Hom(G,C_k),f\bigr)$, where $\bm 0=(0)_{v\in V(G)}$. 
\begin{proposition}\label{prop:p-covering}
    The continuous map $p_f\colon \Conn(D_f,\bm 0)\to \Conn\bigl(\Hom(G,C_k),f\bigr)$ is a covering map.
\end{proposition}
\begin{proof}
    Note that $p_f\colon \Conn(D_f,\bm 0)\to \Conn\bigl(\Hom(G,C_k),f\bigr)$ is the composite of 
    \begin{equation}
    \label{eqn:p-decomposition}
    \Conn(D_f,\bm 0)\xrightarrow{q_f} \Conn(D_f/k'\Z, [\bm 0]_{k'})\xrightarrow{r_f} \Conn\bigl(\Hom(G,C_k),f\bigr).
    \end{equation}
    Since $q_f\colon D_f\to D_f/k'\Z$ is a covering map, the first factor in \cref{eqn:p-decomposition} is a covering map. 
    Since $r_f\colon D_f/k'\Z\to \Hom(G,C_k)$ is a homeomorphism onto a union of connected components in $\Hom(G,C_k)$, the second factor in \cref{eqn:p-decomposition} is a homeomorphism. 
\end{proof}
We shall show in \cref{prop:contractible} that $\Conn(D_f,\bm 0)$ is contractible, which implies that $p_f\colon \Conn(D_f,\bm 0)\to \Conn\bigl(\Hom(G,C_k),f\bigr)$ is the universal cover of $\Conn\bigl(\Hom(G,C_k),f\bigr)$.

\section{Contractibility of the universal cover}
\label{sec:Contractibility}
In order to show the contractibility of $\Conn(D_f,\bm 0)$, we first give a more explicit description of $\Conn(D_f,\bm 0)$. 
Define $E_f^{(0)}\subseteq \Z^{V(G)}$ by
\begin{align*}
E_f^{(0)} =D_f^{(0)}\cap \{\,(a_v)_{v\in V(G)}\in \Z^{V(G)}\mid \text{$a_v=0$ for each $v\in V(G)$ in a directed cycle in $\overrightarrow{G}_f$}\,\}
\end{align*}
and define $E_f$ to be the induced cube subcomplex of $D_f$ with vertex set $E_f^{(0)}$. Our first goal in this section is to show \cref{prop:E-alpha-as-Conn}, which asserts that we have $\Conn(D_f,\bm 0)=E_f$.

\begin{lemma}
\label{prop:L-alpha-and-L-p-x}
    For each $\bm{b}=(b_v)_{v\in V(G)}\in D_f^{(0)}$, we have $-\bm b=(-b_v)_{v\in V(G)}\in D_{p_f(\bm b)}^{(0)}$.
\end{lemma}
\begin{proof}
    Take any $uv\in E(G)$. We may assume $(u,v)\in A(\overrightarrow{G}_f)$ without loss of generality. Then we have $b_v\leq b_u\leq b_v + 1$. 
    If $b_u=b_v$, then $-b_u=-b_v$ and hence $-\bm b$ satisfies the $uv$-th defining inequalities of $D_{p_f(\bm b)}^{(0)}$. Suppose that $b_u=b_v+1$. Then we have $p_f(\bm b) (v)- p_f(\bm b) (u)=-[1]_k$ by \cref{eqn:p-alpha-a-cases}. Therefore we have $(v,u)\in A(\overrightarrow{G}_{p_f(\bm b)})$ and the $uv$-th defining inequalities of $D_{p_f(\bm b)}^{(0)}$ are $x_u\leq x_v\leq x_u+1$, which are satisfied by $-\bm b$.
\end{proof}

\begin{corollary}
\label{cor:symmetry}
    Suppose that $g,h\colon G\to C_k$ are homomorphisms. Then, there exists $\bm{b}\in D_g^{(0)}$ such that $p_g(\bm b) =h$ if and only if there exists $\bm{c}\in D_h^{(0)}$ such that $p_h(\bm c) =g$. 
\end{corollary}
\begin{proof}
    Suppose that there exists $\bm{b}\in D_g^{(0)}$ such that $p_g(\bm b) =h$.
    Then by \cref{prop:L-alpha-and-L-p-x}, we have $-\bm b\in D_{p_g(\bm b)}^{(0)}=D_{h}^{(0)}$. Clearly we have $p_h(-\bm b)=g$. 
\end{proof}

\begin{lemma}
\label{lem:directed-cycles-remain-invariant}
    For each $\bm a=(a_v)_{v\in V(G)}\in D_f^{(0)}$, the directed cycles in $\overrightarrow{G}_f$ and in $\overrightarrow{G}_{p_f(\bm a)}$ coincide.
\end{lemma}
\begin{proof}
    By \cref{cor:symmetry}, it suffices to show that each directed cycle in $\overrightarrow{G}_f$ is a directed cycle in $\overrightarrow{G}_{p_f(\bm a)}$. 
    Thus let $(v_0,v_1,\dots, v_{\ell-1},v_\ell=v_0)$ be a directed cycle in $\overrightarrow{G}_f$. Then $\bm a$ satisfies $a_{v_{i}}\leq a_{v_{i-1}}\leq a_{v_i}+1$ for each $i\in\{1,\dots, \ell\}$. This forces $a_{v_0}=a_{v_1}=\dots=a_{v_{\ell-1}}=a_{v_\ell}$. Therefore $(v_0,v_1,\dots, v_{\ell-1},v_\ell=v_0)$ is a directed cycle in $\overrightarrow{G}_{p_f(\bm a)}$. 
\end{proof}

\begin{corollary}
\label{cor:Conn-subset-E}
    We have $\Conn(D_f,\bm 0)\subseteq E_f$.
\end{corollary}
\begin{proof}
    \cref{prop:local-str-of-L,lem:directed-cycles-remain-invariant} imply that, 
    if $u\in V(G)$ is in a directed cycle in $\overrightarrow{G}_f$, then the value of the $u$-th coordinate remain invariant in any connected component in $D_f$. 
    (In the context of combinatorial reconfiguration, one says that $u$ is \emph{frozen} with respect to $f$; cf.~\cite[Lemma~5.1]{Wro20}.
    For example, in the situation of \cref{fig:G-f}, each homomorphism $g\colon G\to C_5$ in $\Conn\bigl(\Hom(G,C_5),f\bigr)$ takes the same value as $f$ at any of the vertices in the directed $5$-cycle in $\overrightarrow{G}_f$.)
\end{proof}

For each $\bm a =(a_v)_{v\in V(G)}\in\Z^{V(G)}$, define the nonnegative integer $\norm{\bm a}$ as $\sum_{v\in V(G)}|a_v|$.

\begin{proposition}
\label{prop:E-alpha-as-Conn}
    For any $\bm a\in E^{(0)}_f$ with $\bm a \neq \bm 0$, there exists $\bm a'\in E^{(0)}_f$ which is adjacent to $\bm a$ and satisfies $\norm{\bm a'}=\norm{\bm a}-1$. 
    In particular, we have $\Conn(D_f,\bm 0)=E_f$. 
\end{proposition}
\begin{proof}
    Since we have $\Conn(D_f,\bm 0)\subseteq E_f\subseteq D_f$ by \cref{cor:Conn-subset-E}, in order to show $\Conn(D_f,\bm 0)=E_f$, it suffices to show that $E_f$ is connected, or equivalently, that the graph $E_f^{(1)}$ is connected. Therefore the second assertion indeed follows from the first.

    Take any $\bm a\in E_f^{(0)}$ such that $\bm a\neq \bm 0$. 
    Suppose that there exists $v\in V(G)$ with $a_v > 0$; the case where there exists $v\in V(G)$ with $a_v < 0$ is similar. 
    Observe that if $(u,v)$ is an arc in the digraph $\overrightarrow{G}_{p_f(\bm a)}$, then we have $a_u\leq a_v$. Indeed, $(u,v)\in A(\overrightarrow{G}_{p_f(\bm a)})$ means $uv\in E(G)$ and 
    \[
    f(v)-f(u) + [2a_v -2a_u]_k=[1]_k
    \]
    in $\Z/k\Z$.
    Since $\bm a\in E_f^{(0)}$ implies that $|a_v - a_u| \le 1$, 
    if $f(v)-f(u)=[1]_k$, then we have $a_v-a_u=0$, and if $f(v)-f(u)=-[1]_k$, then we have $a_v-a_u=1$ (recall that $k \ge 3$ and $k \neq 4$). 
    
    Let $S\subseteq V(G)$ be the set of vertices $v\in V(G)$ such that $a_v>0$. What we have just shown implies that if $u\in S$ and $(u,v)\in A(\overrightarrow{G}_{p_f(\bm a)})$, then $v\in S$. It follows that there exists a vertex $u\in S$ which is a sink in $\overrightarrow{G}_{p_f(\bm a)}$; otherwise, we would be able to find a directed cycle in $\overrightarrow{G}_{p_f(\bm a)}$ contained in $S$, which contradicts $\bm a\in E_f^{(0)}$ by \cref{lem:directed-cycles-remain-invariant}. Let $u\in S$ be such a vertex. 
    Then we have $\bm a -\bm e_u\in E_f^{(0)}$ by \cref{prop:local-str-of-L}. 
    Clearly $\bm a -\bm e_u$ and $\bm a$ are adjacent and  $\norm{\bm a -\bm e_u} = \norm{\bm a}-1$. 
\end{proof}

Now we are ready to show that $\Conn(D_f,\bm 0)=E_f$ is contractible. 

\begin{lemma}
\label{lem:maximal-cube-below}
    Let $\bm a\in D_f^{(0)}$. Define $S_-=\{\,v\in V(G)\mid \text{$a_v<0$ and $\bm a+\bm e_v\in D_f^{(0)}$}\,\}$ and $S_+=\{\,v\in V(G)\mid \text{$a_v>0$ and $\bm a-\bm e_v\in D_f^{(0)}$}\,\}$. Then, for each $J_-\subseteq S_-$ and $J_+\subseteq S_+$, we have $\bm a+\sum_{v\in J_-} \bm e_v -\sum_{v\in J_+}\bm e_v\in D_f^{(0)}$.
\end{lemma}
\begin{proof}
    It suffices to show that $S_-\cup S_+\subseteq V(G)$ is an independent set of $G$ (i.e., no two vertices in $S_-\cup S_+\subseteq V(G)$ are adjacent), because then all defining inequalities of $D_f^{(0)}$ for $\bm a+\sum_{v\in J_-} \bm e_v -\sum_{v\in J_+}\bm e_v$ are identical to those for one of $\bm a$, $\bm a + \bm e_v$ with $v\in S_-$, or $\bm a - \bm e_v$ with $v\in S_+$, and hence are satisfied. 
    
    First suppose $u,v\in S_-$ and $uv\in E(G)$. Without loss of generality, we may assume that $(u,v)\in A(\overrightarrow{G}_f)$. Then the $uv$-th defining inequalities of $D^{(0)}_f$ for $\bm a +\bm e_u$ and $\bm a + \bm e_v$ assert $a_v\leq a_u+1\leq a_v+1$ and $a_v+1\leq a_u\leq a_v+2$, respectively. But they imply $a_u+1\leq a_u$, a contradiction. 
    Similarly, the existence of $u,v\in S_+$ with $uv\in E(G)$ leads to a contradiction.
    
    Finally, observe that if $u\in S_-$ and $v\in S_+$, then we cannot have $uv\in E(G)$, since otherwise $\bm a$ cannot satisfy the $uv$-th defining inequalities of $D^{(0)}_f$ (which entail $|a_v-a_u|\leq 1$).
\end{proof}

Let $C=[0,1]^d$ be a cube and $\bm a\in \{0,1\}^d$ a vertex of $C$. We define $\partial_{\bm a} C\subseteq C$ as $\bigcup_{i=1}^{d}[0,1]^{i-1}\times \{1-a_i\}\times [0,1]^{d-i}$. 
The cube $C$ strongly deformation retracts onto $\partial_{\bm a}C$.
For example, when $\bm a=\bm 1=(1,\dots, 1)$, we have  $\partial_{\bm 1}C=\bigcup_{i=1}^d[0,1]^{i-1}\times\{0\}\times[0,1]^{d-i}$ (see \cref{fig:flag}) and there is a linear homotopy $H\colon C\times [0,1]\to C$ defined by
\[
H(\bm c,t)=\bm{c}-t\min\{c_1,\dots,c_d\}\bm 1,
\]
giving rise to a strong deformation retract of $C$ onto $\partial_{\bm 1}C$. 

\begin{figure}[h]
    \centering
		\begin{tikzpicture}[scale=0.4]
			\coordinate(000) at (0,0);
			\coordinate(100) at (5,0);
			\coordinate(010) at (2,2);
			\coordinate(110) at (7,2);
			\coordinate(001) at (0,5);
			\coordinate(101) at (5,5);
			\coordinate(011) at (2,7);
			\coordinate(A01) at (2,5);
			\coordinate(A02) at (5,2);
			
			\fill [gray!20] (A01) -- (011) -- (001) -- cycle;
			\fill [gray!40] (A01) -- (010) -- (A02) -- (101) --cycle;
			\fill [gray!20] (A02) -- (100) -- (110) -- cycle;
			\fill [gray!60] (A01) -- (001) -- (000) -- (100) -- (A02) -- (010) -- cycle;

			\draw [thick] (000) -- (100) -- (110) -- (010) -- cycle;
			\draw [thick] (000) -- (010) -- (011) -- (001) -- cycle;
			\draw [thick] (000) -- (001) -- (101) -- (100) -- cycle;
		\end{tikzpicture}
    \caption{$\partial_{\bm 1}C$ with $d=3$.}
    \label{fig:flag}
\end{figure}

\begin{proposition}
\label{prop:contractible}
The topological space $E_f$ is contractible.
\end{proposition}
\begin{proof}
For each nonnegative integer $\ell$, define $E_{f,\ell}^{(0)}= \{\,\bm a\in E^{(0)}_f\mid \norm{\bm a}\leq\ell \,\}$ and define $E_{f,\ell}$ as the induced cube subcomplex of $E_f$ with vertex set $E_{f,\ell}^{(0)}$. We construct a (strong) deformation retract of $E_{f,\ell+1}$ onto $E_{f,\ell}$ for each nonnegative integer $\ell$.
This implies that each $E_{f,\ell}$ is homotopy equivalent to $E_{f,0}=\{\bm 0\}$.
Since every compact subset of $E_f$ is contained in some $E_{f,\ell}$, it follows that all the higher homotopy groups of $E_f$ vanish. Hence  $E_f$ is contractible by Whitehead's theorem (see, e.g.,~\cite[Theorem~4.5]{Hat02}).

Let $\ell$ be a nonnegative integer and take any $\bm a\in E_{f,\ell+1}^{(0)}$ with $\norm{\bm a}=\ell+1$. (If there is no such $\bm a$, then we have $E_{f,\ell}=E_{f,\ell+1}$.) Then by \cref{lem:maximal-cube-below}, there exists a largest cube $C$ in $E_{f,\ell+1}$ having $\bm a$ as a vertex, and we can apply a strong deformation retract of $C$ onto $\partial_{\bm a} C$. Notice that $\partial_{\bm a} C\subseteq E_{f,\ell}$ holds since the dimension of $C$ is greater than $0$ by the first statement of \cref{prop:E-alpha-as-Conn}. 
Performing this deformation for each $\bm a\in E_{f,\ell+1}^{(0)}$ with $\norm{\bm a}=\ell+1$, we obtain the required (strong) deformation retract of $E_{f,\ell+1}$ onto $E_{f,\ell}$.
\end{proof}

\section{Determining the homotopy type}
\label{sec:Determining}
We can now determine the homotopy type of $\Conn\bigl(\Hom(G,C_k),f\bigr)$.

\begin{theorem}
\label{thm:homotopy-type-of-Hom}
    If the digraph $\overrightarrow{G}_f$ contains a directed cycle, then $\Conn\bigl(\Hom(G,C_k),f\bigr)$ is contractible. Otherwise, $\Conn\bigl(\Hom(G,C_k),f\bigr)$ is homotopy equivalent to a circle.
\end{theorem}
\begin{proof}
    First suppose that $\overrightarrow{G}_f$ contains a directed cycle. We claim that $p_f\colon E_f\to \Conn\bigl(\Hom(G,C_k),f\bigr)$ is a homeomorphism in this case.
    To see this, it suffices to show that if $\bm a,\bm b\in E_f^{(0)}$ satisfy $p_f(\bm a)=p_f(\bm b)$, then $\bm a= \bm b$. Let $S\subseteq V(G)$ be the set of all vertices $v\in V(G)$ with $a_v=b_v$. $S$ is nonempty since it contains all directed cycles in $\overrightarrow{G}_f$. 
    Moreover, if we have $uv\in E(G)$ and $a_u=b_u$, then $a_v=b_v$ holds.
    This is because $a_v$ is determined by $a_u$ and $p_f(\bm a)$ by \cref{eqn:p-alpha-a-cases}, and similarly for $b_v$.
    Therefore we have $S=V(G)$ by the connectivity of $G$. It follows that $\Conn\bigl(\Hom(G,C_k),f\bigr)$ is contractible since $E_f$ is. 

    Next suppose that $\overrightarrow{G}_f$ does not contain a directed cycle. In this case, we have $E_f=D_f$. We claim that the fundamental group of $\Conn\bigl(\Hom(G,C_k),f\bigr)$ is isomorphic to the additive group $\Z$ of integers. Since the fundamental group of $\Conn\bigl(\Hom(G,C_k),f\bigr)$ is isomorphic to the automorphism group $\Aut(p_f)$ of the universal cover $p_f\colon D_f\to \Conn\bigl(\Hom(G,C_k),f\bigr)$ (see, e.g.,~\cite[Proposition~1.39]{Hat02}), 
    we show that the latter is isomorphic to $\Z$. 
    We obtain a group homomorphism $h\colon \Z\to \Aut(p_f)$ by mapping each $m\in\Z$ to the automorphism $mk'\cdot(-)\colon D_f\to D_f$, where $k'$ is the integer defined in \cref{sec:Construction}. 
    Since the injectivity of $h$ is clear, it suffices to show that $h$ is surjective. 
    Let $\sigma\in \Aut(p_f)$. For any $\bm a\in D_f^{(0)}$ and any $v\in V(G)$, we can write $\sigma(\bm a)_v= a_v+m(\bm a,v)k'$ for some $m(\bm a,v)\in\Z$. It suffices to show that $m(\bm a,v)$ is independent of $\bm a$ and $v$. 
    We have $m(\bm a,u)=m(\bm a, v)$ whenever $uv\in E(G)$ because we must have $|a_v-a_u|\leq 1$ and $|\sigma(\bm a)_v-\sigma(\bm a)_u|\leq 1$, while $k'\geq 3$. We have $m(\bm a,v)=m(\bm b,v)$ whenever $\bm a,\bm b\in D_f^{(0)}$ are adjacent, because we have $|b_v-a_v|\leq 1$ and $|\sigma(\bm b)_v-\sigma(\bm a)_v|\leq 1$, while $k'\geq 3$. 
    Therefore $\Conn\bigl(\Hom(G,H),f\bigr)$ is an Eilenberg--MacLane space $K(\Z,1)$, and hence is homotopy equivalent to a circle. 
\end{proof}

\begin{remark}
More precisely, by a strong form of Whitehead's theorem found in \cite[Theorem~4.5]{Hat02}, $\Hom(G,C_k)$ admits a strong deformation retract onto some points and embedded circles.
\end{remark}

\begin{proof}[Proof of \cref{thm:HomGCk}]
    Combine \cref{prop:square,prop:isolated,thm:homotopy-type-of-Hom}.
\end{proof}


\begin{thebibliography}{10}

\bibitem{BaKo06}
E.~Babson and D.~N. Kozlov.
\newblock Complexes of graph homomorphisms.
\newblock {\em Israel J. Math.}, 152:285--312, 2006.

\bibitem{BaKo07}
E.~Babson and D.~N. Kozlov.
\newblock Proof of the {L}ov\'{a}sz conjecture.
\newblock {\em Ann. of Math. (2)}, 165(3):965--1007, 2007.

\bibitem{Bjo84}
A.~Bj{\"o}rner.
\newblock Posets, regular {CW} complexes and {B}ruhat order.
\newblock {\em European J. Combin.}, 5(1):7--16, 1984.

\bibitem{CuKo06}
S.~L. {\v C}uki\'{c} and D.~N. Kozlov.
\newblock The homotopy type of complexes of graph homomorphisms between cycles.
\newblock {\em Discrete Comput. Geom.}, 36(2):313--329, 2006.

\bibitem{Die17}
R.~Diestel.
\newblock {\em Graph theory}, volume 173 of {\em Graduate Texts in Mathematics}.
\newblock Springer, Berlin, \!\!, fifth edition, 2017.

\bibitem{DoSi23}
A.~Dochtermann and A.~Singh.
\newblock Homomorphism complexes, reconfiguration, and homotopy for directed graphs.
\newblock {\em European J. Combin.}, 110:Paper No. 103704, 31, 2023.

\bibitem{FKN25}
S.~Fujii, K.~Kimura, and Y.~Nozaki.
\newblock Homotopy types of {H}om complexes of graph homomorphisms whose codomains are square-free.
\newblock arXiv:2412.19039v1, 2024.

\bibitem{Hat02}
A.~Hatcher.
\newblock {\em Algebraic topology}.
\newblock Cambridge University Press, Cambridge, \!\!, 2002.

\bibitem{ito2011}
T.~Ito, E.~D. Demaine, N.~J.~A. Harvey, C.~H. Papadimitriou, M.~Sideri, R.~Uehara, and Y.~Uno.
\newblock On the complexity of reconfiguration problems.
\newblock {\em Theoretical Computer Science}, 412:1054--1065, 2011.

\bibitem{Koz08}
D.~Kozlov.
\newblock {\em Combinatorial algebraic topology}, volume~21 of {\em Algorithms and Computation in Mathematics}.
\newblock Springer, Berlin, \!\!, 2008.

\bibitem{LMS25}
B.~L{\'e}v{\^e}que, M.~M{\"u}hlenthaler, and T.~Suzan.
\newblock Reconfiguration of digraph homomorphisms.
\newblock {\em SIAM Journal on Discrete Mathematics}, 39(1):327--360, 2025.

\bibitem{Lov78}
L.~Lov\'{a}sz.
\newblock Kneser's conjecture, chromatic number, and homotopy.
\newblock {\em J. Combin. Theory Ser. A}, 25(3):319--324, 1978.

\bibitem{Mat17JMSUT}
T.~Matsushita.
\newblock Fundamental groups of neighborhood complexes.
\newblock {\em J. Math. Sci. Univ. Tokyo}, 24(3):321--353, 2017.

\bibitem{Mat25}
T.~Matsushita.
\newblock Hom complexes of graphs whose codomains are square-free.
\newblock arXiv:2412.19144v1, 2024.

\bibitem{Wro20}
M.~Wrochna.
\newblock Homomorphism reconfiguration via homotopy.
\newblock {\em SIAM J. Discrete Math.}, 34(1):328--350, 2020.

\end{thebibliography}
\end{document}